\documentclass[11pt,letter]{extarticle}

\usepackage[letterpaper,margin=1in]{geometry}

\usepackage{times,pslatex}
\usepackage{amsfonts}
\usepackage{amssymb}
\usepackage{amsmath}
\usepackage{graphicx}
\usepackage[all]{xy}
\usepackage{amsthm}
\usepackage{color}
\usepackage{xcolor}
\usepackage{bm}
\usepackage[section]{algorithm}
\usepackage{appendix}

  \RequirePackage{enumerate} 
          \usepackage[shortlabels]{enumitem}
\usepackage[colorlinks, linkcolor=reference,
 citecolor=citation,
         urlcolor=e-mail]{hyperref}
\usepackage[numbers]{natbib}

\RequirePackage{verbatim}

\RequirePackage{hyperref}

\RequirePackage{color}

\definecolor{conditional}{rgb}{0,1,0}
\definecolor{e-mail}{rgb}{0,.40,.80}
\definecolor{reference}{rgb}{.20,.60,.22}
\definecolor{mrnumber}{rgb}{.80,.40,0}
\definecolor{citation}{rgb}{.20,.60,.22}

\usepackage{bbm}

\usepackage{tikz}
\usetikzlibrary{arrows,decorations.markings}

\newtheorem{theorem}{Theorem}[section]

\newtheorem{lemma}[theorem]{Lemma}
\newtheorem{proposition}[theorem]{Proposition}
\newtheorem{corollary}[theorem]{Corollary}
\theoremstyle{definition}
\newtheorem{notation}[theorem]{Notation}
\newtheorem{definition}[theorem]{Definition}
\newtheorem{remark}[theorem]{Remark}
\newtheorem{example}[theorem]{Example}

\DeclareMathOperator{\rank}{rank}

\DeclareMathOperator{\init}{in}

\DeclareMathOperator{\Wr}{Wr}

\DeclareMathOperator{\lead}{\text{lead}}
\newcommand{\whu}{\hat{u}}

\newcommand{\ZZ}{\mathbb{Z}}
\newcommand{\cinfty}{\CC^{\infty}(0)}
\newcommand{\Q}{\mathbb{Q}}

\newcommand{\CC}{\mathbb{C}}
\newcommand{\E}{\mathcal{E}}
\newcommand{\F}{\mathcal{F}}

\newcommand{\Frac}{\text{Frac}}

\usepackage{enumitem}
\setlist[enumerate]{leftmargin=.5in}
\setlist[itemize]{leftmargin=.5in}

\title{Parameter identifiability and input-output equations}

\author{Alexey Ovchinnikov\thanks{Department of Mathematics, CUNY Queens College and Ph.D. Programs in Mathematics and Computer Science, CUNY Graduate Center, New York, USA; e-mail: \href{mailto:aovchinnikov@qc.cuny.edu}{aovchinnikov@qc.cuny.edu},}\and Gleb Pogudin\footnote{LIX, CNRS, \'Ecole Polytechnique, Institute Polytechnique de Paris, France and Department of Computer Science, National Research University Higher School of Economics, Moscow, Russia, and Courant Institute of Mathematical Sciences, New York University, New York, USA; e-mail: \href{mailto:gleb.pogudin@polytechnique.edu}{gleb.pogudin@polytechnique.edu}} \and Peter Thompson\thanks{Ph.D. Program in Mathematics, CUNY Graduate Center, New York, USA; e-mail: \href{mailto:pthompson@gradcenter.cuny.edu}{pthompson@gradcenter.cuny.edu}}}
\date{}
\begin{document}
\maketitle

\begin{abstract}
Structural parameter identifiability is a property of a differential model with parameters that allows for the parameters to be determined from the model equations in the absence of noise. 
One of the standard approaches to assessing this problem is via input-output equations and, in particular, characteristic sets of differential ideals.
The precise relation between identifiability and input-output identifiability is subtle.
The goal of this note is to clarify this relation.
The main results are:
\begin{itemize}
    \item identifiability implies input-output identifiability;
    \item these notions coincide if the model does not have rational first integrals;
    \item the field of input-output identifiable functions is generated by the coefficients of a ``minimal'' characteristic set of the corresponding differential ideal.
\end{itemize}
We expect that some of these facts may be known to the experts in the area, but we are not aware of any articles in which these facts are stated precisely and rigorously proved.
\end{abstract}

\section{Introduction}
Structural identifiability  is a property of an ODE model with parameters that allows for the parameters to be uniquely determined from the model equations in the absence of noise. Performing  identifiablity analysis is an important first step in evaluating and, if needed, adjusting the model before a reliable practical parameter identification (determining the numerical values of the parameters) is performed. Details on different approaches to assessing  identifiability can be found, for example, in  \cite{comparison,Hong, Villaverde2019}, which also contain additional references showing practical relevance of studying structural identifiabilty in biological models, from animal sciences to oncology.

 In more detail but still roughly speaking, a function of parameters in an ODE model is identifiable if, generically, two different values of the function result in two different values of the output of the model. A preciese formulation of this concept is given in Definition~\ref{def:seid}. These functions of parameters could be just the parameters themselves, in which case we consider the more standard notion of identifiability of individual parameters. However, it could happen that all of the parameters are not identifiable but some non-trivial functions of the parameters are. Since identifiability is a desirable property to have, finding identifiable functions of non-identifiable parameters could be helpful in reparametrizing the model so that the new model has fewer non-identifiable parameters (see an intentionally simple  Example~\ref{ex:same} to illustrate this issue).

Input-output equations have been used to assess structural identifiability for three decades already going back to~\cite{OllivierPhD}, 
and several prominent software packages are based on this approach~\cite{BD16, SADA2003,Meshkat14res,Meshkat18,BEC2013,DAISY,DAISYIFAC,DAISYMED,COMBOS,COMBOS2,MRS2016}.
However, it has been known that input-output identifiability is not always the same as identifiability (\cite[Example~2.16]{Hong}, \cite[Section 5.2~and 5.3]{allident}).
The goal of this note is to state and prove basic facts about these relations, some of which seem to be implicitly assumed in the current literature.
The main results are
\begin{itemize}
    \item identifiability implies input-output identifiability (Theorem~\ref{prop:definclusion});
    \item these notions coincide if the model does not have rational first integrals (Theorem~\ref{prop:no_first_integrals});
    \item the field of input-output identifiable functions is generated by the coefficients of a ``minimal'' characteristic set of the corresponding differential ideal (Corollary~\ref{cor:charsetIO}).
\end{itemize}

The paper is organized as follows. We begin by stating an analytic definition of identifiabily and algebraic definition of input-output identifiability and show a few simple examples comparing these two not equivalent notions in Section~\ref{sec:ident_general}. In Section~\ref{sec:algcrit}, we prove a technical result, an algebraic criterion for identifiability of functions in terms of field extensions, which is typically much easier to use than the analytic Definition~\ref{def:seid}. In Section~\ref{sec:implication}, we establish theoretical connections between identifiability and input-output identifiability. We finish with Section~\ref{sec:charsetcomput}, in which we prove that input-output identifiability can be computed with characteristic sets from differential algebra, introducing the corresponding mathematical background and notation there.


\section{General definition of identifiability}\label{sec:ident_general}

\subsection{Identifiability}
Fix positive integers $\lambda$, $n$, $m$, and $\kappa$ for the remainder of the paper.
Let $\bm{\mu} = (\mu_1,\ldots,\mu_\lambda)$, $\mathbf{x} = (x_1, \ldots, x_n)$, $\mathbf{y} = (y_1, \ldots, y_m)$, and $\mathbf{u} = (u_1,\ldots,u_\kappa)$.  These are called the parameters, the state variables, the outputs, and the inputs, respectively.
Consider a system of ODEs \hypertarget{Sigma}{}
\begin{equation}\label{eq:sigma}
\Sigma = 
\begin{cases}
\mathbf{x}' = \cfrac{\mathbf{f}(\mathbf{x}, \bm{\mu}, \mathbf{u})}{Q(\mathbf{x}, \bm{\mu}, \mathbf{u})}, \\
\mathbf{y} = \cfrac{\mathbf{g}(\mathbf{x}, \bm{\mu}, \mathbf{u})}{Q(\mathbf{x}, \bm{\mu}, \mathbf{u})},\\
\mathbf{x}(0) = \mathbf{x}^\ast,
\end{cases}
\end{equation}
where $\mathbf{f} = (f_1, \ldots, f_n)$ and $\mathbf{g} = (g_1, \ldots, g_m)$ are tuples of elements of $\CC[\bm{\mu},\mathbf{x},\mathbf{u}]$ and $Q\in\CC[\bm{\mu},\mathbf{x},\mathbf{u}] \backslash \{0\}$.

\begin{notation}[Auxiliary analytic notation]
\begin{enumerate}[label = (\alph*),leftmargin=7.5mm]
\item[] 
\item Let $\CC^\infty(0)$ denote the set of all functions that
    are complex analytic in some neighborhood of~$t=0$.
    \item A subset $U\subset \CC^\infty(0)$ is called \emph{Zariski open} if there
exist $h \in \ZZ_{\geqslant 0}$ and a non-zero 
polynomial~$P(u_0, u_1, \ldots, u_h)\in\CC[u_0, \ldots, u_h]$ such that
\[
U=\big\{\whu \in\cinfty\mid P\big(\whu, \whu^{(1)}, \ldots, \whu^{(h)}\big)|_{t=0}\neq0\big\}.
\]
\item Let 
    $\tau(\CC^{s})$
        denote the set of all Zariski open non-empty
    subsets of $\CC^{s}$ and
    $\tau(\CC^\infty(0))$
denote the set of all Zariski open non-empty subsets
of $\CC^\infty(0)$.
    \item Let $\Omega=\{(\hat{\mathbf{x}}^*,\hat{\bm{\mu}},\hat{\mathbf{u}}) \in \CC^n \times \CC^\lambda \times (\CC^\infty(0))^\kappa \mid Q(\hat{\mathbf{x}}^*,\hat{\bm{\mu}},\hat{\mathbf{u}}(0)) \ne 0$\}  and 
    \[
    \Omega_h = \Omega \cap (\{(\hat{\mathbf{x}}^\ast, \hat{\bm{\mu}}) \in \CC^{n + \lambda} \mid h(\hat{\mathbf{x}}^\ast, \hat{\bm{\mu}}) \text{ well-defined}\} \times (\CC^\infty(0))^\kappa)
    \]
    for every given $h \in \CC(\mathbf{x}^\ast, \bm{\mu})$.
    \item For $(\hat{\mathbf{x}}^\ast, \hat{\bm{\mu}}, \hat{\mathbf{u}}) \in \Omega$, let $X(\hat{\mathbf{x}}^\ast, \hat{\bm{\mu}}, \hat{\mathbf{u}})$ and $Y(\hat{\mathbf{x}}^\ast, \hat{\bm{\mu}}, \hat{\mathbf{u}})$ denote the unique solution over $\CC^\infty(0)$ of the instance of $\Sigma$ with $\mathbf{x}^\ast = \hat{\mathbf{x}}^\ast$, $\bm{\mu} = \hat{\bm{\mu}}$, and $\mathbf{u} = \hat{\mathbf{u}}$ (see~\cite[Theorem~2.2.2]{Hille}).
    
\end{enumerate}
\end{notation}
 Definition~\ref{def:seid} given below, being a generalization from individual parameters to functions of parameters, is a precise (and  unambiguous) way of expressing the following widely used analytic understanding of the identifiability concept: a  parameter in~\eqref{eq:sigma} is identifiable if generically two different parameter values result in two different values of the output~\cite{CD1980,WL1981,WL1982,V1983,V1984,LG94,WP1996,WP1997,MRCW2001,DAISY,SADA2003,MXPW2011}. A discussion on the comparison can be found in \cite[Remark~2.6]{Hong}. The complexity of the presentation of Definition~\ref{def:seid} is the price to pay for being precise. 
\begin{definition}[Identifiability, see {\cite[Definition~2.5]{Hong}}]\label{def:seid}
We say that $h(\mathbf{x}^\ast,   \bm{\mu}) \in \CC(\mathbf{x}^\ast,  \bm{\mu})$ is \emph{identifiable} if 
\begin{gather*}
  \exists \Theta \in \tau(\CC^n \times \CC^\lambda) \; \exists U \in \tau((\CC^\infty(0))^\kappa)\;\\
  \forall (\hat{\mathbf{x}}^\ast, \hat{\bm{\mu}}, \hat{\mathbf{u}}) \in (\Theta \times U) \cap \Omega_h \quad |S_h(\hat{\mathbf{x}}^\ast, \hat{\bm{\mu}}, \hat{\mathbf{u}})| = 1,
\end{gather*}
where
\[
  S_h(\hat{\mathbf{x}}^\ast, \hat{\bm{\mu}}, \hat{\mathbf{u}}) := \{h(\tilde{\mathbf{x}}^\ast,  \tilde{\bm{\mu}}) \mid  (\tilde{\mathbf{x}}^\ast, \tilde{\bm{\mu}}, \hat{\mathbf{u}}) \in \Omega_h  \;\text{ and }\; Y(\hat{\mathbf{x}}^\ast, \hat{\bm{\mu}}, \hat{\mathbf{u}}) = Y(\tilde{\mathbf{x}}^\ast, \tilde{\bm{\mu}}, \hat{\mathbf{u}}) \}.
\]
In this paper, we are interested in comparing identifiability and IO-identifiability (Definition~\ref{def:ioid}), and the latter is defined 
for functions in $\bm{\mu}$, not in $\bm{\mu}$ and $\mathbf{x}^\ast$.
Thus, just for the purpose of comparison, we will restrict ourselves to the field \[\{h \in \mathbb{C}(\bm{\mu})\mid h \text{ is identifiable}\},\] which we will call \emph{the field of identifiable functions}.
\end{definition}

\begin{remark}
  The above definition can be extended to functions $h(\mathbf{x}^\ast,  \bm{\mu}) \in \CC(\mathbf{x}^\ast,  \bm{\mu})$ (see Definition~\ref{def:seid}).
  There are software tools that can assess identifiability of initial conditions (e.g., SIAN~\cite{SIAN}).
  Any such tool can be used to assess identifiability of a given function $h(\mathbf{x}^\ast,  \bm{\mu}) \in \CC(\mathbf{x}^\ast,  \bm{\mu})$ by means of the transformation described in~\eqref{eq:identifiable_function} in the proof of Proposition~\ref{prop:idfrac}.
\end{remark}

\subsection{IO-identifiability}

\begin{notation}[Differential algebra]
\begin{enumerate}[label = (\alph*),leftmargin=7.5mm]
  \item A {\em differential ring} $(R,\delta)$ is a commutative ring with a derivation $':R\to R$, that is, a map such that, for all $a,b\in R$, $(a+b)' = a' + b'$ and $(ab)' = a' b + a b'$. 
  \item The {\em ring of differential polynomials} in the variables $x_1,\ldots,x_n$ over a field $K$ is the ring $K[x_j^{(i)}\mid i\geqslant 0,\, 1\leqslant j\leqslant n]$ with a derivation defined on the ring by $(x_j^{(i)})' := x_j^{(i + 1)}$. 
  This differential ring is denoted by $K\{x_1,\ldots,x_n\}$.
  \item An ideal $I$ of a differential ring $(R,\delta)$ is called a {\em differential ideal} if, for all $a \in I$, $\delta(a)\in I$. For $F\subset R$, the smallest differential ideal containing set $F$ is denoted by $[F]$.
    \item \hypertarget{colon}{}\hypertarget{infinity}{}
  For an ideal $I$ and element $a$ in a ring $R$, we denote $I \colon a^\infty = \{r \in R \mid \exists \ell\colon a^\ell r \in I\}$.
  This set is also an ideal in $R$.
  \hypertarget{ISigma}{}
    \item Given $\Sigma$ as in~\eqref{eq:sigma}, we define the differential ideal of $\Sigma$ as $I_\Sigma=[Q\mathbf{x}'-\mathbf{f},Q\mathbf{y}-\mathbf{g}]:Q^\infty \subset \CC(\bm{\mu})\{\mathbf{x},\mathbf{y},\mathbf{u}\}$.
\end{enumerate}
\end{notation}

 The following definition of IO-identifiability captures the most probable, in our opinion, actual intent of prior attempts of defining and computing it via a characteristic set of the prime differential ideal of $I_\Sigma$  \cite{COMBOS,DAISY,DAISYMED,DAISYIFAC}.

For a subclass of models, called linear compartment models, for each output variable, an explicit linear algebra-based formula was proposed in~\cite{Meshkat14res} to find IO-equations to determine IO-identifiability. In general, using these equations instead of the just mentioned characteristic set-based approach would give incorrect results  (see~\cite[Remark~3.11]{Meshkat18}).
However,  \cite[Theorem~3]{OPT19} shows that such an approach is valid  for a large class of linear compartment models. 

We will see in Corollary~\ref{cor:charsetIO} that characteristic sets (more precisely, characteristic presentations) provide a tool of computing IO-identifiability. However, for the purposes of mathematical elegance  and a more explicit connection with other branches of mathematics, e.g., with model theory
(which was recently discovered to be useful for identifiability \cite{second_paper}), we present a definition that is short and avoids  notationally heavy definitions leading to characteristic sets: 

\begin{definition}[IO-identifiability]\label{def:ioid}\hypertarget{fieldio}{}
The smallest field $k$ such that $\CC \subset k \subset \CC(\bm{\mu})$ and  $\hyperlink{ISigma}{I_\Sigma} \cap \CC(\bm{\mu})\{\mathbf{y},\mathbf{u}\}$ is generated (as an ideal or as a differential ideal) by $I_\Sigma \cap k\{\mathbf{y},\mathbf{u}\}$ is called \emph{the field of IO-identifiable functions}.
We call $h \in \CC(\bm{\mu})$  \emph{IO-identifiable} if $h \in k$.
\end{definition}

We will now briefly compare Definitions~\ref{def:seid} and~\ref{def:ioid} by considering  intentionally simple  examples.

\begin{example}\label{ex:same}
Consider the system
\[
\Sigma =
\begin{cases}
x'=(a+b)x\\
y=x.
\end{cases}
\]
So, $\lambda = 2$, $n=m=1$, and $\kappa = 0$. Let us check the identifiability of $h_1(x_1^\ast,a,b) = a$. As there are no denominators, $Q = 1$, and so $\Omega_h = \CC^3$. Let $\Theta \in \tau(\CC^3)$ and $(\hat x_1^\ast,\hat a,\hat b) \in \Theta$. Then
\[
Y(\hat x_1^\ast,\hat a,\hat b) = \hat x_1^\ast e^{(\hat a+\hat b)t}.\]
Hence,
\begin{align*}
S_{h_1}(\hat x_1^\ast,\hat a,\hat b)&=\big\{h_1(\tilde x_1^\ast,\tilde a,\tilde b)\mid (\tilde x_1^\ast,\tilde a,\tilde b) \in \CC^3\ \ \text{and}\ \ Y(\hat x_1^\ast,\hat a,\hat b)=Y(\tilde x_1^\ast,\tilde a,\tilde b)\big\}\\
&=\big\{\tilde a \in \CC \mid \exists(\tilde x_1^\ast,\tilde b) \in \CC^2\ \ \text{such that}\ \ Y(\hat x_1^\ast,\hat a,\hat b)=Y(\tilde x_1^\ast,\tilde a,\tilde b)\big\}\\
&\supset\big\{\tilde a \in \CC \mid \exists(\tilde x_1^\ast,\tilde b) \in \CC^2\ \ \text{such that}\ \ \hat x_1^\ast = \tilde x_1^\ast\ \text{and}\ \hat a+\hat b = \tilde a + \tilde b\big\}\\
&=\big\{\tilde a \in \CC \mid \exists\,\tilde b \in \CC\ \ \text{such that}\ \ \hat a+\hat b - \tilde a = \tilde b\big\} = \CC,
\end{align*}
therefore, by Definition~\ref{def:seid},  $h_1 = a$ is not identifiable.
We will now check the identifiability of $h_2(x_1^\ast,a,b)=a+b$.  Let \[\Theta=\left\{(x_1^\ast,a,b)\in\CC^3\mid x_1^\ast \ne 0\right\}\in \tau(\CC^3)\] and consider any $(\hat x_1^\ast,\hat a,\hat b) \in \Theta$. We have
\begin{align*}
S_{h_2}&(\hat x_1^\ast,\hat a,\hat b)=\big\{h_2(\tilde x_1^\ast,\tilde a,\tilde b)\mid (\tilde x_1^\ast,\tilde a,\tilde b) \in \CC^3\ \ \text{and}\ \ Y(\hat x_1^\ast,\hat a,\hat b)=Y(\tilde x_1^\ast,\tilde a,\tilde b)\big\}\\
&=\big\{\tilde a+\tilde b \mid (\tilde a,\tilde b)\in\CC^2 \ \text{such that}\ \exists\,\tilde x_1^\ast \in \CC\ \ \text{such that}\ \ Y(\hat x_1^\ast,\hat a,\hat b)=Y(\tilde x_1^\ast,\tilde a,\tilde b)\big\}\\
&=\big\{\tilde a+\tilde b \mid (\tilde a,\tilde b)\in\CC^2 \ \text{such that}\ \exists\,\tilde x_1^\ast \in \CC\ \ \text{such that}\ \ \hat x_1^\ast = \tilde x_1^\ast\ \text{and}\ \hat a+\hat b = \tilde a + \tilde b\big\}\\
&=\{\hat a +\hat b\}.
\end{align*}
Therefore, by Definition~\ref{def:seid}, $h_2= a+b$ is identifiable. With this conclusion, it is now natural to consider the following reparametrization of $\Sigma$:
\[
\begin{cases}
x'=cx\\
y=x,
\end{cases}
\]
in which the only parameter $c$ is identifiable.
This shows how considering identifiable functions of parameters rather than just the parameters could be helpful for improving the model.

We will now investigate the IO-identifiability of $\Sigma$. We have \[I_\Sigma=[x'-(a+b)x,y-x]=[y'-(a+b)y,x-y]\subset\CC(a,b)\{x,y\}.\]
Hence $I_\Sigma\cap \CC(a+b)\{y\}$ generates $I_\Sigma \cap \CC(a,b)\{y\} = [y'-(a+b)y]$. 
So $k \subset \CC(a + b)$.
On the other hand,  $I_\Sigma \cap \CC(a, b)[y, y']$ is the principal ideal generated by $f := y' - (a + b)y$  because, for instance, the non-zero solution of~$\Sigma$ being an exponential function, is not algebraic over the constants. Hence, since $f$ has one of the coefficients equal $1$, $k$ must contain $a + b$, so $k = \CC(a + b)$. We will later see in Corollary~\ref{cor:charsetIO} how one can avoid considering actual solutions to find input-output identifiable functions.
\end{example}

We will now consider an example in which the identifiability and IO-identifiability do not coincide.
\begin{example}\label{ex:IOidnotid}
Consider an example of a twisted harmonic oscillator:
\[ \begin{cases}
 x_1' = (\omega + \alpha)x_2\\
x_2' = -\omega x_1\\
  y = x_2
\end{cases}
\]
in which $\alpha$ can be measured separately, so is assumed to be known. 
This can be reflected as follows:
\[ \Sigma =\begin{cases}
 x_1' = (\omega + x_3)x_2\\
x_2' = -\omega
x_1\\
x_3' = 0
\\
  y_1 = x_2\\
  y_2 = x_3.
\end{cases}
\]
So, $\lambda = 1$, $n=3$, $m=2$, and $\kappa =0$. Let $h(x_1^\ast,x_2^\ast,x_3^\ast,\omega) = \omega$, so we are checking the identifiability of $\omega$. As there are no denominators, $Q = 1$, and so $\Omega_h = \CC^4$.
Let $\Theta \in \tau(\CC^4)$ and $(\hat x_1^\ast,\hat x_2^\ast,\hat x_3^\ast,\widehat\omega) \in \Theta$ be such that  $\widehat \omega (\hat x_3^\ast + \widehat \omega) \neq 0$ and $\hat x_3^\ast \neq -2 \widehat \omega$.
Then, denoting the frequency by $\widehat{\varphi} := \sqrt{\widehat{\omega} (\hat{x}_3^\ast + \widehat{\omega})}$, we have
\[
Y(\hat x_1^\ast,\hat x_2^\ast,\hat x_3^\ast,\widehat\omega)=\begin{pmatrix}
 -\hat x_1^\ast\frac{ \widehat\omega }{\widehat{\varphi}} \sin\left(t \widehat{\varphi}\right) + \hat x_2^\ast\cos\left(t \widehat{\varphi}\right)\\
\hat x_3^\ast
\end{pmatrix}.
\]
Note that $Y(\hat x_1^\ast,\hat x_2^\ast,\hat x_3^\ast,\widehat\omega)(0) = \begin{pmatrix}
\hat x_2^\ast\\
\hat x_3^\ast
\end{pmatrix}$.
Therefore,
\begin{align*}
S_h&(\hat x_1^\ast,\hat x_2^\ast,\hat x_3^\ast,\widehat\omega) \\ &=\left\{h(\tilde x_1^\ast,\tilde x_2^\ast,\tilde x_3^\ast,\widetilde\omega)\mid (\tilde x_1^\ast,\tilde x_2^\ast,\tilde x_3^\ast,\widetilde\omega)\in \CC^4\ \  \text{and}\ \ Y(\hat x_1^\ast,\hat x_2^\ast,\hat x_3^\ast,\widehat\omega)=Y(\tilde x_1^\ast,\tilde x_2^\ast,\tilde x_3^\ast,\widetilde\omega)\right\}\\
&=\left\{\widetilde\omega \in \CC\mid \exists (\tilde x_1^\ast,x_2^\ast,x_3^\ast) \in \CC^3\  \ \text{such that}\ \ Y(\hat x_1^\ast,\hat x_2^\ast,\hat x_3^\ast,\widehat\omega)=Y(\tilde x_1^\ast,\tilde x_2^\ast,\tilde x_3^\ast,\widetilde\omega)\right\}\\
&=\left\{\widetilde\omega \in \CC\mid \exists\, \tilde x_1^\ast\in \CC\ \ \text{such that}\ \ Y(\hat x_1^\ast,\hat x_2^\ast,\hat x_3^\ast,\widehat\omega)=Y(\tilde x_1^\ast,\hat x_2^\ast,\hat x_3^\ast,\widetilde\omega)\right\}\\
& \supset\left\{\widetilde\omega \in \CC\:\Big|\:\exists\, \tilde x_1^\ast\in \CC\ \ \text{s.t.}\ \ \hat x_1^\ast\widehat\omega = \tilde x_1^\ast\widetilde\omega \  \ \text{and}\ \ \widehat\omega(\hat x_3^\ast + \widehat\omega) = \widetilde\omega(\hat{x}_3^\ast + \widetilde\omega)\right\}, 
\end{align*}
which has cardinality $2$ because the second conjunct has distinct solutions $\widetilde \omega \in \{\widehat \omega, -(\hat x_3^\ast + \widehat\omega)\}$ and by the first conjunct $\tilde x_1^\ast$ is uniquely determined by the choice of $\widetilde \omega$.
Therefore, by Definition~\ref{def:seid}, $\omega$ is not identifiable.

On the other hand, 
\[
  I_\Sigma = [x_1'-(\omega+x_3)x_2,x_2'+\omega x_1,x_3',y_1-x_2,y_2-x_3] \subset \mathbb{C}(\omega)\{x_1,x_2,x_3,y_1,y_2\}.
\]

One can verify that $I_\Sigma \cap \CC(\omega)[y_1, y_1', y_2] = \{0\}$.
Indeed, if there were a polynomial $p(\omega, y_1, y_1', y_2) \ne 0$ in this intersection, then, for every solution of $\Sigma$, the evaluation at $t = 0$ would imply that $p(\widehat{\omega}, \hat{x}_2^\ast, -\widehat{\omega}\hat{x}_1^\ast, \hat{x}_3^\ast) = 0$ yielding that $\hat{x}_1^\ast, \hat{x}_2^\ast, \hat{x}_3^\ast$, and $\widehat{\omega}$ always satisfy  such a polynomial relation.
But this is not the case because they can be chosen to be any complex numbers.
Therefore, $I_\Sigma \cap \CC(\omega)[y_1, y_1', y_1'', y_2]$ is a principal ideal generated by $f := y_1'' + \omega^2y_1 + \omega y_1y_2$.
Since $f$ has one of its coefficients equal to $1$, the field $k$ from Definition~\ref{def:ioid} must contain $\omega$, so $k = \CC(\omega)$. In particular, $\omega$ is input-output identifiable (but is not identifiable).
A more systematic way of computing this field using characteristic sets, as described in Corollary~\ref{cor:charsetIO} and shown in Example~\ref{ex:5}.

\end{example}


\section{Technical result: algebraic criterion for identifiability}\label{sec:algcrit}
Proposition~\ref{prop:idfrac} extends the algebraic criterion for identifiability~\cite[Proposition~3.4]{Hong} to identifiability of functions of parameters rather than identifiability of just specific parameters themselves.

\begin{proposition}\label{prop:idfrac}
For every $h \in \CC(\mathbf{x}^\ast, \bm{\mu})$, the following are equivalent:
\begin{itemize}[leftmargin=7.5mm]
\item  $h$ is identifiable;
\item the image of $h$ in $\Frac(\CC(\bm{\mu})\{\mathbf{x},\mathbf{y},\mathbf{u}\}/I_\Sigma)$ lies in the field generated by the image of $\CC\{\mathbf{y},\mathbf{u}\}$ in $\Frac(\CC(\bm{\mu})\{\mathbf{x},\mathbf{y},\mathbf{u}\}/I_\Sigma)$.
\end{itemize}
\end{proposition}

\begin{example}
For $\Sigma$ from Example~\ref{ex:same}, we have $I_\Sigma = [x'-(a+b)x,y-x]$, and see that 
\begin{align*}
L := \Frac(\CC(a,b)\{x,y\}/I_\Sigma) &= \Frac(\CC(a,b)\{x,y\}/[x'-(a+b)x,y-x] \\
&=\Frac(\CC(a,b)\{x,y\}/[y'-(a+b)y,y-x]. 
\end{align*}
Hence, the field of fractions of the image of $\CC(a,b)\{y\}$ in $L$ is  \[M :=\Frac\left(\CC(a,b)\{y\}/[y'-(a+b)y]\right).\]
Since $a+b = y'/y$, we have $h_2 = a+b \in M$, and so $h_2$ is identifiable by Proposition~\ref{prop:idfrac}.
\end{example}

\begin{proof}
  Write $h = h_1/h_2$, where $h_1,h_2 \in \CC[\mathbf{x}^\ast, \bm{\mu}]$.  
  Let $\F = \Frac(\CC(\bm{\mu})\{\mathbf{x},\mathbf{y},\mathbf{u}\}/I_\Sigma)$ and  
  $\E$ the subfield generated by the image of $\CC\{\mathbf{y},\mathbf{u}\}$  in $\F$.  
  Let $\Sigma_1$ be the system of equations obtained by adding
  \begin{equation}\label{eq:identifiable_function}
  \begin{aligned}
      &x_{n+1}' = \sum\limits_{i = 1}^n f_i \frac{\partial h}{\partial x_i},\\ 
      &y_{m+1} = x_{n+1} - h,\\
      &x_{n+1}(0) = x_{n+1}^\ast     
  \end{aligned}
  \end{equation} 
  to \hyperlink{Sigma}{$\Sigma$}, where $x_{n+1}$ is a new state variable and $y_{m+1}$ is a new output.
 Note that $x_{n + 1}' = h'$ and $y_{m + 1}' = 0$.
  We define 
  \[
    \F_1 = \Frac(\CC(\bm{\mu})\{\mathbf{x}, x_{n + 1}, \mathbf{y}, y_{m + 1},\mathbf{u}\}/I_{\Sigma_1}),
  \] 
  and let $\E_1$ be the subfield generated by the image of $\CC\{\mathbf{y}, y_{m + 1}, \mathbf{u}\}$ in $\F_1$.
  We will talk about $\Sigma$-identifiability of $h$ and $\Sigma_1$-identifiability of $x_{n + 1}^\ast$. 
  The proof will proceed in the following three steps.

  \begin{enumerate}[align=left,itemindent=0.08in,leftmargin=0in, label = \textbf{Step~\arabic*}.]
     \item\label{step:ident} \underline{\emph{$h$ is $\Sigma$-identifiable $\iff$ $x_{n + 1}^\ast$ is $\Sigma_1$-identifiable.}}
     Assume that $h$ is $\Sigma$-identifiable. 
     Let $\Theta$ and $U$ be the corresponding open subsets from Definition~\ref{def:seid}.
     We set 
     \[
        \Theta_1 := \{ (\hat{\mathbf{x}}^\ast, \hat{x}_{n + 1}^\ast, \hat{\bm{\mu}}) \mid (\hat{\mathbf{x}}^\ast, \hat{\bm{\mu}}) \in \Theta \;\&\; h_2(\hat{\mathbf{x}}^\ast, \hat{\bm{\mu}}) \neq 0\}.
     \]
     We will show that $x_{n + 1}^\ast$ is identifiable with the open sets from Definition~\ref{def:seid} being $\Theta_1$ and~$U$.
     Let $\Omega_1$ be the set $\Omega$ for the model $\Sigma_1$, and consider 
     $(\hat{\mathbf{x}}^\ast, \hat{x}_{n + 1}^\ast, \hat{\bm{\mu}}, \hat{\mathbf{u}}) \in (\Theta_1 \times U) \cap \Omega_1$.
     Since, for a fixed known value of $y_{m + 1}$, the values of $x_{n + 1}^\ast$ and $h(\mathbf{x}^\ast, \bm{\mu})$ uniquely determine each other, we have
     \[
      |S_{x_{n + 1}^\ast}(\hat{\mathbf{x}}^\ast, \hat{x}_{n + 1}^\ast, \hat{\bm{\mu}}, \hat{\mathbf{u}})| = |S_{h}(\hat{\mathbf{x}}^\ast, \hat{\bm{\mu}}, \hat{\mathbf{u}})| = 1.
     \]
     Thus, $x_{n + 1}^\ast$ is $\Sigma_1$-identifiable.

     For the other direction, assume that $x_{n + 1}^\ast$ is $\Sigma_1$-identifiable, and $\Theta_1$ and $U_1$ are the corresponding open sets from Definition~\ref{def:seid}.
     Let $\Theta$ be the projection of $\Theta_1$ onto all of the coordinates except for $x_{n + 1}^\ast$.
     We will show that $h$ is $\Sigma$-identifiable with the open sets being $\Theta$ and $U_1$.
     Consider $(\hat{\mathbf{x}}^\ast, \hat{\bm{\mu}}, \hat{\mathbf{u}}) \in (\Theta \times U_1) \cap \Omega_h$.
     Let $\hat{x}_{n + 1}^\ast \in \CC$ be such that $(\hat{\mathbf{x}}^\ast, \hat{x}_{n + 1}^\ast, \hat{\bm{\mu}}) \in \Theta_1$.
     Then, using the fact that $y_{m + 1}$ is constant so is equal to its initial condition, we have
     \begin{align*}
        &1  = |S_{x_{n + 1}^\ast} (\hat{\mathbf{x}}^\ast, \hat{x}_{n + 1}^\ast, \hat{\bm{\mu}}, \hat{\mathbf{u}})| \\
        &= |\{ \tilde{x}_{n + 1}^\ast \mid (\tilde{\mathbf{x}}^\ast, \tilde{\bm{\mu}}, \hat{\mathbf{u}}) \in \Omega_h,\; \tilde{x}_{n + 1}^\ast \in \CC \text{ and } Y_1(\hat{\mathbf{x}}^\ast, \hat{x}_{n + 1}^\ast, \hat{\bm{\mu}}, \hat{\mathbf{u}}) = Y_1(\tilde{\mathbf{x}}^\ast, \tilde{x}_{n + 1}^\ast, \tilde{\bm{\mu}}, \hat{\mathbf{u}}) \}| \\
        &= \left|\left\{ h(\tilde{\mathbf{x}}^\ast, \tilde{\bm{\mu}}) \:\bigg|\: (\tilde{\mathbf{x}}^\ast, \tilde{\bm{\mu}}, \hat{\mathbf{u}}) \in \Omega_h, \tilde{x}_{n + 1}^\ast \in \CC \text{ and } \begin{cases} 
          Y(\hat{\mathbf{x}}^\ast, \hat{\bm{\mu}}, \hat{\mathbf{u}}) = Y(\tilde{\mathbf{x}}^\ast, \tilde{\bm{\mu}}, \hat{\mathbf{u}}),\\
          \hat{x}^\ast_{n + 1} - h(\hat{\mathbf{x}}^\ast, \hat{\bm{\mu}}) = \tilde{x}^\ast_{n + 1} - h(\tilde{\mathbf{x}}^\ast, \tilde{\bm{\mu}})
        \end{cases} \right\}\right|\\
        &= |S_h(\hat{\mathbf{x}}^\ast, \hat{\bm{\mu}}, \hat{\mathbf{u}})|.
     \end{align*}

     \item\label{step:fields} \underline{\emph{$h \in \E \iff x_{n + 1} \in \E_1$.}}
     Observe that we have natural embeddings $\F \hookrightarrow \F_1$ and $\E \hookrightarrow \E_1$.
     If $h \in \mathcal E$, then $x_{n + 1} = y_{m + 1} + h \in \E_1$.

     Assume that $x_{n + 1} \in \E_1$.
     Then $h = x_{n + 1} - y_{m + 1} \in \E_1$.
     Observe that $\F_1 = \F(x_{n + 1})$, and $x_{n + 1}$ is transcendental over $\F$.
    Since none of the right-hand sides of the equations for the state variables involves $x_{n + 1}$, 
      there is a differential automorphism $\alpha\colon \F_1 \to \F_1$ such that $\alpha(x_{n + 1}) = x_{n + 1} + 1$ and $\alpha|_{\F} = \operatorname{id}$.
     Since $\alpha(y_{m + 1}) = y_{m + 1} + 1$, we have $\alpha(\E_1) \subset \E_1$.
     Since $\E_1 = \E(y_{m + 1})$ and $\alpha(y_{m + 1}) = y_{m + 1} + 1$, every $\alpha$-invariant element of $\E_1$ belongs to $\E$.
     Since $\alpha(h) = h$, we have $h \in \mathcal E$.

     \item  From~\ref{step:ident}, $h$ is identifiable if and only if $x_{n + 1}^\ast$ is $\Sigma_1$-identifiable.
     By \cite[Proposition~3.4 (a) $\iff$ (c); Remark~2.2]{Hong}, $x_{n+1}^\ast$ is $\Sigma_1$-identifiable if and only if $x_{n+1} \in \E_1$.
     Finally, \ref{step:fields} implies that $x_{n + 1} \in \E_1$ if and only if $h \in \E$.
\end{enumerate}
\end{proof}


\section{Identifiability and IO-identifiability}\label{sec:implication}

\subsection{Identifiability $\implies$ IO-identifiability but not the other way around}

\begin{remark}
 We have already seen an ODE model in which all parameters are IO-identifiable but are not identifiable (Example~\ref{ex:IOidnotid}).  Real-life examples of ``slow-fast ambiguity'' in chemical reactions and of a 
Lotka-Volterra model with the same conclusion can be found in~\cite[Sections~5.2 and~5.3]{allident}.
\end{remark}

\begin{theorem}\label{prop:definclusion} 
For all \hyperlink{Sigma}{$\Sigma$} and $h \in\mathbb{C}(\hyperlink{vars}{\bm{\mu}})$,
\[
  h\text{ is identifiable }\implies h\text{ is IO-identifiable } 
\]
\end{theorem}

\begin{proof}
Let $h \in \CC(\bm{\mu})$ be identifiable.  
By Proposition~\ref{prop:idfrac}, there exist $g \in \CC\{\mathbf{y},\mathbf{u}\}\backslash \hyperlink{ISigma}{I_\Sigma}$ and $w \in \CC\{\mathbf{y},\mathbf{u}\}$ such that $gh + w \in I_\Sigma$.  Therefore, there exist $m_1,\ldots,m_r \in \CC(\bm{\mu})\{\mathbf{y},\mathbf{u}\}$ and $p_1,\ldots,p_r \in I_\Sigma \cap \hyperlink{fieldio}{k}\{\mathbf{y},\mathbf{u}\}$ such that 
\begin{equation}\label{eq:quot}
gh+w = m_1p_1+\ldots+m_rp_r.
\end{equation}
Suppose $h \not\in k$.  
By \cite[Theorem 9.29, p. 117]{milneFT}, there exists an automorphism $\sigma$ on $\overline{\CC(\bm{\mu})}$ that fixes $k$ pointwise and such that $\sigma(h) \neq h$.  Let $R_1 := \overline{\CC(\bm{\mu})}\{\mathbf{x},\mathbf{y},\mathbf{u}\}$.  
We extend $\sigma$ to $R_1$ by letting $\sigma$ fix $\mathbf{x}$, $\mathbf{y}$, and $\mathbf{u}$.  Applying $\sigma$ to \eqref{eq:quot} and subtracting the two equations yields
\begin{equation}\label{eq:quot2}
g(h-\sigma(h)) = (m_1-\sigma(m_1))p_1+\ldots+(m_r-\sigma(m_r))p_r
\end{equation}
in $R_1$.  Let $P$ denote the differential ideal generated by $\Sigma$ in $R_1$. 
Since $P$ is a prime differential ideal and the right-hand side of \eqref{eq:quot2} belongs to $P$, it follows that either $g \in P$ or $h-\sigma(h) \in P$.  
But since $h-\sigma(h)$ is a non-zero element of $\overline{\CC(\bm{\mu})}$ and $P$ is a proper ideal, it cannot be that $h-\sigma(h) \in P$.  
Therefore, $g \in P$.  Hence, $g \in P \cap R = I_\Sigma$, contradicting our assumption on $g$.
\end{proof}


\subsection{Sufficient condition for ``$\text{identifiable} \iff \text{IO-identifiable}$''}
The aim of this section is Theorem~\ref{prop:no_first_integrals}, which gives a sufficient condition for the fields of identifiable and IO-identifiable functions to coincide.

\begin{notation}\label{not:basic_diffalg}
\begin{itemize}[leftmargin=7.5mm]
  \item[]
  \item For a differential ring $(R,\delta)$, its ring of {\em constants} is $C(R) := \{r\in R\mid \delta(r)=0\}$.\hypertarget{Wr}{}
  \item For elements $a_1, \ldots, a_N$ of a differential ring, let $\Wr_M(a_1, \ldots, a_N)$ denote the $M \times N$ Wronskian matrix of $a_1, \ldots, a_N$, that is,
  \[
    \Wr_M(a_1, \ldots, a_N)_{i, j} = a_j^{(i - 1)}, \quad 1 \leqslant j \leqslant N,\; 1\leqslant i \leqslant M.
  \]
\end{itemize}
\end{notation}

\begin{definition}[Field of definition]
Let $L \subseteq K$ be fields and let $X$ be a (possibly infinite) set of variables.  Let $I$ be an ideal of $K[X]$.  We say the \emph{field of definition of $I$ over $L$} is the smallest (with respect to inclusion) field $k$, $L \subseteq k \subseteq K$, such that $I$ is generated by $I \cap k[X]$.
\end{definition}

\begin{remark}
For a given $X$ and $I$, the field of definition of $K$ over $\Q$ is what is called the field of definition of $K$ (with no reference to a subfield) in \cite[Definition and Theorem 3.4, p. 55]{marker}.
By \cite[Theorem 3.4]{marker}, for every $K$ and $I$, there is a smallest field $k_0 \subseteq K$ such that $I$ is generated by $I \cap k_0[X]$.  The smallest intermediate field $k$, $L \subseteq k \subseteq K$, such that $I$ is generated by $I \cap k[X]$ is equal to the smallest subfield of $K$ containing $L$ and $k_0$.  Therefore, for every $L$, $K$, and $I$, the field of definition of $I$ over $L$ is well defined.
\end{remark}

\begin{lemma}[{{cf.~\cite[Section~4.1]{DVJBNP01}, \cite[Section~3.4]{MXPW11}, and \cite[Section~V.]{XiaMoog}}}]\label{lem:non0wron}
Let $g \in \hyperlink{ISigma}{I_\Sigma}$ be such that we can write $g = \sum_{i=1}^N a_iz_i$, where $N \geqslant 2$, $a_i \in \CC(\bm{\mu}) \backslash \{0\}$, $a_1 = 1$, and $z_1,\ldots,z_N$ are distinct \hyperlink{monomial}{monomials} in $\CC\{\mathbf{y},\mathbf{u}\}$.
If for some $Z \subsetneq \{z_1,\ldots,z_N\}$ of size $N-1$ it holds that $\det \hyperlink{Wr}{\Wr_{N - 1}}(Z) \not\in \hyperlink{DiffIdeal}{I_\Sigma}$, then $a_i$ is identifiable for all $i=1,\ldots,N$.
\end{lemma}

\begin{proof}
  Suppose $\det\Wr_{N - 1}(z_1,\ldots,z_{t-1},z_{t+1},\ldots,z_N) \not\in I_\Sigma$.
Modulo $I_\Sigma$, we have
\begin{equation}\label{eq:wronsk_original_relation}
    \sum_{i\neq t}\frac{a_i}{a_t}z_i = -z_t
\end{equation}
Since $I_\Sigma$ is a differential ideal, the derivatives of~\eqref{eq:wronsk_original_relation} are also true.  
Differentiating~\eqref{eq:wronsk_original_relation} $N - 2$ times, we obtain the following linear system:
\[
M\left(\frac{a_1}{a_t},\ldots,\frac{a_{t-1}}{a_t},\frac{a_{t+1}}{a_t},\ldots,\frac{a_N}{a_t}\right)^T = -(z_t,\ldots,z_t^{(N-2)})^T,
\]
where $M = \Wr_{N-1}(z_1,\ldots,z_{t-1},z_{t+1},\ldots,z_N)$.
Since $M$ is nonsingular modulo $I_\Sigma$, in $\Frac(\CC(\bm{\mu})\{\mathbf{x}, \mathbf{y}\} / I_\Sigma)$, we have
\[
\left(\frac{a_1}{a_t},\ldots,\frac{a_{t-1}}{a_t},\frac{a_{t+1}}{a_t},\ldots,\frac{a_N}{a_t}\right) = (-z_t,\ldots,z_t^{(N-2)})(M^{-1})^T.
\]
Since the entries of the right-hand side belong to the subfield generated by $\CC\{\mathbf{y}, \mathbf{u}\}$, the entries of the left-hand side are identifiable by Proposition~\ref{prop:idfrac}. Since $a_1 = 1$, $a_t$ is identifiable and it follows that $a_2,\ldots,a_N$ are identifiable.
\end{proof}

\begin{theorem}\label{prop:no_first_integrals}
Assume that model $\Sigma$ does not have rational first integrals (i.e., first integrals that are rational functions in the parameters and state variables), that is, the constants of $\Frac(\CC(\bm{\mu})\{\mathbf{x}, \mathbf{y}, \mathbf{u}\} / \hyperlink{ISigma}{I_\Sigma})$ coincide with $\CC(\bm{\mu})$.
Then, for every $h \in \CC(\bm{\mu})$, 
\[
  h \text{ is identifiable} \iff h \text{ is IO-identifiable}.
\]
\end{theorem}

\begin{proof}
  Proposition~\ref{prop:definclusion} implies that the field of all identifiable functions is contained in the field of all IO-identifiable functions.
  
  Let $J := I_\Sigma \cap \CC(\bm{\mu}) \{\mathbf{y}, \mathbf{u}\}$.
We fix an indexing of differential monomials in $\mathbf{y}$ and $\mathbf{u}$ by $\mathbb{N}$, it defines an $\mathbb{N}$-indexed basis $\mathcal{B}$ of $\CC(\bm{\mu}) \{\mathbf{y}, \mathbf{u}\}$.
  Consider an infinite matrix with each row being an element of a $\CC(\bm{\mu})$-basis of $J$ written as a
  vector in basis $\mathcal{B}$.
  Let $M$ be the  reduced row echelon form of the matrix.
  Notice that, since the original matrix has only finitely many nonzero entries in each row, $M$ also has only finitely many nonzero entries in each row.
  The field of definition of $J$ over $\mathbb{C}$ is contained in the field generated by the entries of $M$. 
  Therefore, it is sufficient to prove that the entries of $M$ are identifiable.
  Consider any row of $M$. It corresponds to a differential polynomial $p \in J$.
  Assume that a proper subset of monomials of $p$ is linearly dependent modulo $J$ over $\mathbb{C}(\bm{\mu})$.
  This dependence yields a polynomial $q \in J$.
  The representation of $q$ in basis $\mathcal{B}$ must be reducible to zero by the rows of $M$.
  However, the reduction of $q$ with respect to $p$ is not zero (as they are not proportional), and the result of this reduction is not reducible by any other row of $M$ by the definition of reduced row echelon form.
  Thus, there is no such $q$.
  Hence, the image of every proper subset of monomials of $p$ in $\Frac(\CC(\bm{\mu})\{\mathbf{x}, \mathbf{y}, \mathbf{u}\} / I_\Sigma)$ is linearly independent over the constants of $\Frac(\CC(\bm{\mu})\{\mathbf{x}, \mathbf{y}, \mathbf{u}\} / I_\Sigma)$.
  Thus, \cite[Theorem 3.7, p. 21]{Kap} implies that the Wronskian of every proper subset of monomials of $p$ does not belong to~$I_\Sigma$.
  Lemma~\ref{lem:non0wron} implies that the coefficients of $p$ are identifiable.
\end{proof}

\begin{example}
System $\Sigma$ from Example~\ref{ex:IOidnotid} has $x_3$ as a rational first integral, and, in this example, function $h=\omega$ is IO-identifiable and is not identifiable.
\end{example}


\section{IO-identifiability via characterstic sets}\label{sec:charsetcomput}

\subsection{Differential algebra preliminaries}\label{sec:diffalg}

We will use the following notation and definitions standard in differential algebra (see, e.g., \cite[Chapter~I]{Kol}, \cite[Chapter~I]{Ritt}, and \cite[Section~2]{Boulier2000}):

\begin{definition}\label{def:diffranking}
  A {\em differential ranking} on $K\{x_1,\ldots,x_n\}$ is a total order $>$ on $X := \{\delta^ix_j\mid i\geqslant 0,\, 1\leqslant j\leqslant n\}$ satisfying:
  \begin{itemize}[leftmargin=7.5mm]
    \item for all $x \in X$, $\delta(x) > x$ and
    \item for all $x, y \in X$, if $x >y$, then $\delta(x) > \delta(y)$.
  \end{itemize}
\end{definition}

It can be shown that a differential ranking on $K\{x_1,\ldots,x_n\}$ is always a well order.

\begin{notation}\hypertarget{leadinit}{}
  For  $f \in K\{x_1,\ldots,x_n\} \backslash K$ and differential ranking $>$,
  \begin{itemize}[leftmargin=7.5mm]
    \item $\lead(f)$ is the element of $\{\delta^ix_j \mid i \geqslant 0, 1 \leqslant j \leqslant n\}$ appearing in $f$ that is maximal with respect to $>$.
    \item The leading coefficient of $f$ considered as a polynomial in $\lead(f)$ is denoted by $\init(f)$ and called the initial of $f$. 
    \item The separant of $f$ is $\frac{\partial f}{\partial\lead(f)}$, the partial derivative of $f$ with respect to $\lead(f)$.
    \item The rank of $f$ is $\rank(f) = \lead(f)^{\deg_{\lead(f)}f}$.
    \item For $S \subset K\{x_1,\ldots,x_n\} \backslash K$, the set of initials and separants of $S$ is denoted by $H_S$.
    \item for $g \in K\{x_1,\ldots,x_n\} \backslash K$, say that $f < g$ if $\lead(f) < \lead(g)$ or $\lead(f) = \lead(g)$ and $\deg_{\lead(f)}f < \deg_{\lead(g)}g$.
  \end{itemize}
\end{notation}
\hypertarget{charset}{}
\begin{definition}[Characteristic sets]
  \begin{itemize}[leftmargin=7.5mm]
    \item For $f, g \in K\{x_1,\ldots,x_n\} \backslash K$, $f$ is said to be reduced w.r.t. $g$ if no proper derivative of $\lead(g)$ appears in $f$ and $\deg_{\lead(g)}f <\deg_{\lead(g)}g$.
    \item 
    A subset $\mathcal{A}\subset K\{x_1,\ldots,x_n\} \backslash K$
    is called {\em autoreduced} if, for all $p \in \mathcal{A}$, $p$ is reduced w.r.t. every  element of $\mathcal A\setminus \{p\}$. 
    One can show that every autoreduced set has at most $n$ elements (like a triangular set but unlike a Gr\"obner basis in a polynomial ring).
    \item Let $\mathcal{A} = \{A_1, \ldots, A_r\}$ and $\mathcal{B} = \{B_1, \ldots, B_s\}$ be autoreduced sets such that $A_1 < \ldots < A_r$ and $B_1 < \ldots < B_s$. 
    We say that $\mathcal{A} < \mathcal{B}$ if
    \begin{itemize}
      \item $r > s$ and $\rank(A_i)=\rank(B_i)$, $1\leqslant i\leqslant s$, or
      \item there exists $q$ such that $\rank(A_q) <\rank(B_q)$ and, for all $i$, $1\leqslant i< q$, $\rank(A_i)=\rank(B_i)$.
    \end{itemize}
    \item An autoreduced subset of the smallest rank of a differential ideal $I\subset K\{x_1,\ldots,x_n\}$
    is called a {\em characteristic set} of $I$. One can show that every non-zero differential ideal in $K\{x_1,\ldots,x_n\}$ has a characteristic set. Note that a characteristic set does not necessarily generate the ideal.
  \end{itemize}
\end{definition}
\hypertarget{charpres}{}
\begin{definition}[Characteristic presentation]\label{def:char_pres}
  \begin{itemize}[leftmargin=7.5mm]
    \item A polynomial is said to be {\em monic} if at least one of its coefficients is $1$. Note that this is how monic is typically used in identifiability analysis and not how it is used in~\cite{Boulier2000}. A set of polynomials is said to be monic if each polynomial in the set is monic.
    \item Let $\mathcal C$ be a characteristic set of a prime differential ideal $P \subset K\{z_1,\ldots,z_n\}$.
    Let $N(\mathcal{C})$ denote the set of non-leading variables of $\mathcal{C}$.
    Then $\mathcal{C}$ is called a {\em characteristic presentation} of $P$ if all initials of $\mathcal C$ belong to $K[N(\mathcal C)]$  and 
   none of the elements of $\mathcal{C}$ has a factor in $K[N(\mathcal C)] \backslash K$.  It follows from ~\cite{Boulier2000} that $P$ has a characteristic presentation.
  \end{itemize}
\end{definition}

\hypertarget{monomial}{}
\begin{definition}[Monomial]
    Let $K$ be a differential field and let $X$ be a set of variables.  An element of the differential polynomial ring $K\{X\}$ is said to be a \emph{monomial} if it belongs to the smallest multiplicatively closed set containing $1$, $X$, and the derivatives of $X$.  An element of the polynomial ring $K[X]$ is said to be a \emph{monomial} if it belongs to the smallest multiplicatively closed set containing $1$ and $X$.
\end{definition}


\subsection{IO-identifiable functions via characteristic presentations}
Corollary~\ref{cor:charsetIO} shows how the field of IO-identifiable functions can be computed via input-output equations.

\begin{proposition}\label{prop:charfod}
Let $L \subseteq K$ be differential fields and let $X$ be a finite set of variables.  
Let $P$ be a prime non-zero differential ideal of $K\{X\}$ such that the ideal generated by $P$ in $\overline{K}\{X\}$ is prime.  
If $\mathcal C$ is a \hyperlink{charpres}{monic characteristic presentation} of $P$, 
then the field of definition of $P$ over $L$ is the field extension of $L$ generated by the coefficients of $\mathcal C$.
\end{proposition}

\begin{proof}
Let $A$ be the set of coefficients of $\mathcal C$ and let $k$ be the field of definition of $P$ over $L$.

\underline{Suppose $A \not\subset k$}.  
Let $P_1$ be the ideal generated by the image of $P$ in $\overline{K}\{X\}$.  
We show that $\mathcal C$ is a monic characteristic presentation for $P_1$.  We have that $\mathcal C$ is a \hyperlink{charset}{characteristic set} for $P_1$.  
Since the initials of $\mathcal{C}$ lie in $K[N(\mathcal{C})]$, they also lie in $\overline{K}[N(\mathcal{C})]$.
The property of not having a factor in the nonleading variables does not depend on the coefficient field as well.
By \cite[Definition 2.6]{Hubert00} and the paragraph thereafter, we have that $P=[\mathcal{C}]:H_\mathcal{C}^\infty$ in $K\{X\}$, and therefore $[\mathcal{C}]:H_\mathcal{C}^\infty \subset P_1$, where the differential ideal operation is taken over $\overline{K}\{X\}$.
Since $\mathcal{C}$ is a characteristic set of $P_1$, the paragraph following \cite[Definition 2.4]{Hubert00} implies that $P_1$ is contained in $[\mathcal{C}] \colon H_{\mathcal{C}}^\infty$, so $P_1 = [\mathcal{C}] \colon H_{\mathcal{C}}^\infty$.
Hence,
\cite[Corollary 1, p. 42]{Boulier2000}, we conclude that $
\mathcal C$ is a monic characteristic presentation for $P_1$.

By \cite[Theorem 9.29, p. 117]{milneFT}, there is an automorphism $\alpha$ of $\overline{K}$ that fixes $k$ but moves some element of $A$.  
Extend $\alpha$ to a differential ring automorphism on $\overline{K}\{X\}$ that fixes $X$.  
We show that $\alpha(\mathcal C)$ is a monic characteristic presentation of $P_1$.
Since the initials of $\mathcal C$ lie in $K[N(\mathcal{C})]$ and no element of $\mathcal C$ has a factor in $K[N(\mathcal{C})] \backslash K$, it follows that the initials of $\alpha(\mathcal C)$ lie in $\overline{K}[N(\alpha(\mathcal{C}))]$ and no element of $\alpha(\mathcal C)$ has a factor in $\overline{K}[N(\alpha(\mathcal{C}))]\backslash\overline{K}$.
Since the rank of $\alpha(\mathcal C )$ is the same as that of $\mathcal C$, it remains to show that $\alpha(\mathcal C ) \subset P_1$.
Let $f \in \mathcal C$.  
Since $P$ is defined over $k$, it follows that $P_1$ is defined over $k$.  
Therefore, there exist $a_i \in k\{X\} \cap P_1$  and $b_i \in \overline{K}\{X\}$ such that $f = \sum_i a_ib_i$.  
Thus,
\[
\alpha(f) = \sum_i a_i \alpha(b_i) \in P_1.
\]
We conclude that $\alpha(\mathcal C) \subset P_1$ and thus is a characteristic set of $P_1$.

We have shown that $\mathcal C$ and $\alpha(\mathcal C)$ are monic characteristic presentations of $P_1$.  By \cite[Theorem 3, p. 42]{Boulier2000}, $\alpha(\mathcal C) = \mathcal C$.  However, since $\alpha$ moves some coefficient appearing in $\mathcal C$, we have a contradiction.  We conclude that our assumption that $A \not\subset k$ is false.

\underline{It remains to show that $k \subseteq L(A)$}.
Let $\{h_i\}_{i\in B}$ be a monic generating set of $P_1$ as an ideal such that, for all $i\in B$ and for all $g \in P_1 \backslash \{h_i\}$, the support of $h_i-g$ is not a proper subset of the support of $h_i$.  
We argue that such a generating set exists.
We describe a map $\phi \colon P_1 \rightarrow \mathcal{P}(P_1)$, where $\mathcal{P}(P_1)$ denotes the power set of $P_1$, such that $\forall b \in P_1$
\begin{itemize}[leftmargin=7.5mm]
\item $b$ belongs to the ideal generated by $\phi(b)$ and
\item $\forall a \in \phi(b) \; \forall d \in P_1 \backslash \{0\}$ the support of $d$ is not a proper subset of the support of $a$.
\end{itemize}
Let $b \in P_1$. Construct $\phi(b)$ recursively as follows.  If there is no element of $P_1 \backslash \{0\}$ whose support is a proper subset of the support of $b$, let $\phi(b) = \{b\}$.  If there is an $a \in P_1 \backslash \{0\}$ whose support is a proper subset of the support of $b$, let $\phi(b) = \phi(a) \cup \phi(b - ca)$, where $c \in \CC$ is such that $b - ca$ has smaller support than $b$.  This completes the construction of $\phi$.  Note that the procedure terminates since for each non-terminal step, the support of each element of the output is smaller than the support of the input.  Let $\{b_i\}_{i\in B_0}$ be a generating set for $P_1$ as an ideal.  Now $\bigcup_{i \in B_0}\phi(b_i)$, after normalization so that each element is monic, has the desired properties.

Fix $i$ and suppose that some coefficient of $h_i$ does not belong to $L(A)$.  
Then by \cite[Theorem 9.29, p. 117]{milneFT}, there is an automorphism $\alpha$ of $\overline{K}$ such that $\alpha$ fixes $L(A)$ and $\alpha(h_i) \neq h_i$.  
Since $h_i$ is monic, we have that $h_i - \alpha(h_i)$ has smaller support than $h_i$.  
Now we show that $h_i - \alpha(h_i) \in P_1$.  
Since $h_i \in P_1$, we have that $h_i \in \hyperlink{DiffIdeal}{[}\mathcal C\hyperlink{DiffIdeal}{]}\hyperlink{colon}{:}{\hyperlink{leadinit}{H_{\mathcal C}}}^{\hyperlink{infinity}{\infty}}$.  Therefore, since $\alpha$ fixes the coefficients of $\mathcal C$, we have \[\alpha(h_i) \in [\mathcal C]:H_{\mathcal C}^\infty.\]  
Hence, \[h_i-\alpha(h_i) \in [\mathcal C]:H_{\mathcal C}^\infty = P_1.\]  
This contradicts the definition of $\{h_i\}_{i\in B}$. 
Since the coefficients of $h_i$ belong to $L(A)$, $\{h_i\}_{i\in B}$ is also a generating set for $P$.  
Therefore, $P$ is generated by $P \cap L(A)\{X\}$. 
By the definition of $k$, it follows that $k \subseteq L(A)$.
\end{proof}

 The following result reduces the problem of finding the field of IO-identifiable functions of~\eqref{eq:sigma} to the problem of finding a monic characteristic presentation of the corresponding prime differential ideal. The latter problem has been solved (see e.g.,~\cite{Boulier2000}), and there is an implementation of the corresponding algorithm, Rosenfeld-Gr\"obner, in {\sc Maple}.

\begin{corollary}\label{cor:charsetIO}
If $\mathcal C$ is a \hyperlink{charpres}{monic characteristic presentation} of $\hyperlink{ISigma}{I_\Sigma} \cap \CC(\bm{\mu})\{\mathbf{y},\mathbf{u}\}$, then the field of IO-identifiable functions (as in Definition~\ref{def:ioid}) is generated over $\CC$ by the coefficients of the elements of $\mathcal C$.
\end{corollary}

\begin{proof}
The proof of \cite[Lemma~3.2]{Hong} shows that both $I_\Sigma$ and the ideal generated by the image of $I_\Sigma$ in $\overline{\CC(\bm{\mu})}\{\mathbf{x},\mathbf{y},\mathbf{u}\}$ are prime, since the argument does not depend on the coefficient field.  Therefore $I_\Sigma \cap \CC(\bm{\mu})\{\mathbf{y},\mathbf{u}\}$ and the ideal generated by $I_\Sigma \cap \CC(\bm{\mu})\{\mathbf{y},\mathbf{u}\}$ in $\overline{\CC(\bm{\mu})}\{\mathbf{y},\mathbf{u}\}$ are prime.  By Proposition~\ref{prop:charfod} with $L = \CC$, $K = \CC(\bm{\mu})$, and $P = I_\Sigma \cap \CC(\bm{\mu})\{\mathbf{y},\mathbf{u}\}$, we have that the field of definition of $P$ over $\CC$ is equal to the field extension of $\CC$ generated by the coefficients of $\mathcal C$.  This is exactly the field of IO-identifiable functions.
\end{proof}

\begin{example}\label{ex:5} Consider the following ODE model
\[
\Sigma = \begin{cases}
x_1'  = 0\\
x_2' = x_1x_2 + \mu_1x_1 + \mu_2\\
y = x_2
\end{cases}
\]
As shown in \cite[Lemma~5.1]{second_paper}, neither $\mu_1$ nor $\mu_2$ are identifiable (which can be seen by observing that adding $1$ to $\mu_1$ and subtracting $x_1$ from $\mu_2$ at the same time changes the parameters but does not change the output) and, moreover, the field of identifiable functions is just $\mathbb{C}$.

On the other hand, let us use Corollary~\ref{cor:charsetIO} to compute the field of IO-identifiable functions. We enter these equations in {\sc Maple} and set the elimination differential ranking on the differential variables with $x_1 > x_2 > y$. Within a second, we receive the following characteristic presentation of $I_\Sigma$:
\[
C = \{yx_1 + \mu_1x_1  - y' + \mu_2,\ \ x_2 - y,\ \ yy'' + \mu_1y''  - y'^2 + \mu_2y'\}.
\]
Hence, 
\[\mathcal{C} =C \cap \mathbb{C}(\mu_1,\mu_2)\{y\}= \{yy'' + \mu_1y''  - y'^2 + \mu_2y'\}.\]
 By Corollary~\ref{cor:charsetIO}, the field of IO-identifiable functions is $\mathbb{C}(\mu_1,\mu_2)$, which is also not equal to the field of identifiable functions.
\end{example}

\section*{Acknowledgments}We are grateful to the CCiS at CUNY Queens College for the computational resources and to Julio Banga, Marisa Eisenberg, Nikki Meshkat, Maria Pia Saccomani, Anne Shiu, Seth Sullivant, and Alejandro Villaverde for useful discussions.
We also thank the referees for their comments, which helped us  improve the manuscript.
This work was partially supported by the NSF grants CCF-1563942, CCF-1564132, CCF-1319632, DMS-1760448, CCF-1708884, DMS-1853650, DMS-1853482; NSA grant \#H98230-18-1-0016; and CUNY grants PSC-CUNY \#69827-0047, \#60098-00 48.
\bibliographystyle{abbrvnat}
\setlength{\bibsep}{5.5pt}
\bibliography{bibdata}
\end{document}